\documentclass[journal,a4paper]{IEEEtran} 
\usepackage{stmaryrd}
\usepackage{amsfonts}
\usepackage{graphicx,times,amsmath}
\usepackage{graphics,color}
\usepackage{graphicx}
\usepackage{latexsym}
\usepackage{amsfonts}
\usepackage{amssymb}
\usepackage{url}
\usepackage{subfigure}
\usepackage{times}
\usepackage{color}
\usepackage{cite}
\newtheorem{theorem}{Theorem}
\newtheorem{lemma}{Lemma}

\newtheorem{corollary}{Corollary}
\newtheorem{definition}{Definition}

\newtheorem{remark}{Remark}

\newcommand{\oneton}{1,\cdots,n}

\newcommand{\Rn}{\mathbb{R}^{n}}
\newcommand{\Rnn}{\mathbb{R}^{n\times n}}

\DeclareMathOperator{\diag}{diag} 
 
 \graphicspath{/figures}

\DeclareMathOperator{\var}{\mathrm{var}}

\begin{document}

\title{Pinning consensus in networks of multiagents via a single impulsive controller}

\author{Bo~Liu, Wenlian Lu and~
        Tianping~Chen,~\IEEEmembership{Senior Member,~IEEE}
\thanks{This work is jointly supported by the National Natural
Sciences Foundation of China under Grant Nos. 61273211, 60974015,
and 61273309, the Foundation for the Author of National Excellent
Doctoral Dissertation of P.R. China No. 200921, the Shanghai
Rising-Star Program (No. 11QA1400400), the Marie Curie International
Incoming Fellowship from the European Commission
(FP7-PEOPLE-2011-IIF-302421), and China Postdoctoral Science
Foundation No. 2011M500065.}
\thanks{Bo Liu is with School of Mathematical Sciences, Fudan
University, Shanghai, People's Republic of China, and Institute of
Industrial Science, The University of Tokyo, 4-6-1 Komaba,
Meguro-ku, Tokyo 153-8505, Japan (liu7bo9@gmail.com).}
\thanks{Wenlian Lu is with Centre for Computational Systems Biology, and Laboratory of Mathematics for
Nonlinear Science, School of Mathematical Sciences, Fudan
University, Shanghai, People's Republic of China
(wenlian@fudan.edu.cn).}
\thanks{Tianping Chen is with the School of Mathematical Sciences,
Fudan University, 200433, Shanghai, China. (Corresponding author,
email: tchen@fuan.edu.cn).  } }

\maketitle

\begin{abstract}
\boldmath {\bf In this paper, we discuss pinning consensus in
networks of multiagents via impulsive controllers. In particular, we
consider the case of using only one impulsive controller. We provide
a sufficient condition to pin the network to a prescribed value. It
is rigorously proven that in case the underlying graph of the
network has spanning trees, the network can reach consensus on the
prescribed value when the impulsive controller is imposed on the
root with appropriate impulsive strength and impulse intervals.
Interestingly, we find that the permissible range of the impulsive
strength completely depends on the left eigenvector of the graph
Laplacian corresponding to the zero eigenvalue and the pinning node
we choose. The impulses can be very sparse, with the impulsive
intervals being lower bounded. Examples with numerical simulations
are also provided to illustrate the theoretical results.}
\end{abstract}

\begin{keywords}
consensus, synchronization, multiagent systems, impulsive pinning
control.
\end{keywords}

\section{Introduction}
Coordinated and cooperative control of teams of autonomous systems
has received much attention in recent years. Significant research
activity has been devoted to this area. In the cooperation, group of
agents seek to reach agreement on a certain quantity of interest.
This is the so-called {\em consensus} problem, which has a long
history in computer science. Recently, consensus problem reappeared
in the cooperative control of multi-agent systems and has gained
renewed interests due to the broad applications of multi-agent
systems. A great deal of papers have addressed this problem. For a
review of this area, see the surveys \cite{Ren_2005,Olfati_2007} and
references therein.

The basic idea of consensus is that each agent updates its state
based on the states of its neighbors and its own such that the
states of all agents will converge to a common value. The
interaction rule that specifies the information exchange between an
agent and its neighbors is called the consensus algorithm.

The following is an example of continuous-time consensus algorithm:
\begin{eqnarray}\label{eqnPulse}
\dot{x}_{i}(t)=\sum_{j=1,j\ne
i}^{n}a_{ij}[x_{j}(t)-x_{i}(t)],i=1,\cdots,n
\end{eqnarray}
where $x_{i}(t)\in {\mathbb R}$ is the state of agent $i$ at time
$t$, $a_{ij}\ge 0$ for $i\ne j$ is the coupling strength from agent
$j$ to agent $i$.

Let $a_{ii}=-\sum_{j=1,j\ne i}^{n}a_{ij}$ for $i=1,2,\cdots,n$, we
can have
\begin{eqnarray}\label{eqnPulsea}
\dot{x}_{i}(t)=\sum_{j=1}^{n}a_{ij}x_{j}(t),~i=1,\cdots,n.
\end{eqnarray}


A topic closely related to consensus is synchronization, which can
be written as the following Linearly Coupled Ordinary Differential
Equations (LCODEs):
\begin{eqnarray}
\frac{d x^{i}(t)}{dt}=f(x^{i}(t),t)+c\sum\limits_{j=1}^{n}a_{ij}
x^{j}(t),\quad i=\oneton\label{synn}
\end{eqnarray}
where $x^{i}(t)\in {\mathbb R}^{m}$ is the state variable of the
$i$th node at time $t$, $f$: ${\mathbb
R}^{m}\times[0,+\infty)\rightarrow \mathbb{R}^{m}$ is a continuous
map, $A=[a_{ij}]\in {\mathbb R}^{n\times n}$ is the coupling matrix
with zero-sum rows and $a_{ij}\ge 0$, for  $i\ne j$, which is
determined by the topological structure of the LCODEs.

There are lots of papers discussing synchronization in various
circumstances.

It is clear that the consensus is a special case of synchronization
($f=0$, $m=1$). Therefore, all the results concerning
synchronization can apply to consensus.

It was shown in \cite{Lu2004,Lu2006} that under some assumptions, we
have
\begin{eqnarray}
\lim_{t\rightarrow\infty}||x^{i}(t)-\sum_{j=1}^{n}\xi_{j}x^{j}(t)||=0,~~i=1,\cdots,n,
\end{eqnarray}
where $[\xi_{1},\cdots,\xi_{n}]^{\top}$ is the left eigenvector of
$A$ corresponding to the eigenvalue $0$ satisfying
$\sum_{j=1}^{n}\xi_{j}=1$.

Since in the consensus model,
\begin{eqnarray}
\sum_{j=1}^{n}\xi_{j}x_{j}(t)=\sum_{j=1}^{n}\xi_{j}x_{j}(0)
\end{eqnarray}
for all $t>0$, we have
\begin{eqnarray}
\lim_{t\rightarrow\infty}|x_{i}(t)-\sum_{j=1}^{n}\xi_{j}x_{j}(0)|=0,~~i=1,\cdots,n.
\end{eqnarray}

It can be seen that the agreement value
$\sum_{j=1}^{n}\xi_{j}x_{j}(0)$ strongly depends on the initial
value, which means that the agreement value of the consensus
algorithm is neutral stable (or semi-stable used in some papers).
The concept of neutral stability is used in physics and other
research fields. For example, the principal subspace extraction
algorithms and principal component extraction algorithms discussed
in \cite{Pca,Psa}. A set of equilibrium points is called neutral
stable for a system, if each equilibrium is Lyapunov stable, and
every trajectory that starts in a neighborhood of an equilibrium
converges to a possibly different equilibrium. Similarly, a set of
manifolds is called neutral stable for a system, if every manifold
is invariant, and when there is a small perturbation, the state will
stay in another manifold and never return.

In \cite{Pca}, the manifold discussed is neutral stable, and if the
algorithm is restricted to the manifold, the equilibrium is stable.
Instead, in \cite{Psa}, the equilibrium is neutral stable and the
Stiefel manifold is stable.

In consensus algorithm, the consensus manifold $S=\{x\in
\mathbb{R}^{n}:~x_{1}=x_{2}=\cdots=x_{n}\}$ is the set of
equilibrium points, which is stable. Instead, every point $x\in S$
is neutral stable.

However, in some cases, it is desired that all states converge to a
prescribed value, say, some $s\in \mathbb{R}$. For example, in a
military system, if one wants to use a missile network to attack
some object of the enemy, then it is required that all the missiles
from different military bases should finally hit the same point (see
\cite{JHill_SICON_2011}). Generally, for this purpose, one can make
every state $x_{i}(t)$ converge to $s$ by imposing a negative
feedback term $-[x_{i}(t)-s]$ to agent $i$. However, due to the
interaction of the network, it is not necessary to impose
controllers on all the nodes. This is the basic idea of the pining
control technique, which is an effective class of control schemes.
Generally, in a pinning control scheme, we only need to impose
controllers on a small fraction of the nodes. This is a big
advantage because in large complex networks, it is usually difficult
if not impossible to add controllers to all the nodes. Recently,
pinning strategies have been used in the control of dynamical
networks. For example, decentralized adaptive pinning strategies
have been proposed in \cite{Porfiri_Chaos_2011,SuH_TSMC_2012} for
controlled synchronization of complex networks. And pinning
consensus algorithms have been proposed in
\cite{ChenFei_Automatica_2009,WangZD_TNN_2011}.

Most works on pinning control consider pining a fraction of the
nodes. However, there are a few works that consider pinning only one
node. In \cite{Chen2007}, it was proved that if $\varepsilon>0$, the
following coupled network with a single controller
\begin{eqnarray}
\left\{\begin{array}{cc}\displaystyle\frac{dx^{1}(t)}{dt}&=f(x^{1}(t),t)+c\sum\limits_{j=1}^{n}a_{1j}x^{j}(t)\\
&-c\varepsilon[x^{1}(t)-s(t)],\\
\displaystyle\frac{dx^{i}(t)}{dt}&=f(x^{i}(t),t)+c\sum\limits_{j=1}^{n}a_{ij}x^{j}(t),\\
&i=2,\cdots,n \end{array}\right. \label{pin}
\end{eqnarray}
can pin the complex dynamical network $(\ref{synn})$ to $s(t)$, if
$c$ is chosen suitably. Therefore, the following coupled network
with a single controller
\begin{align}\label{pin1}
\left\{
\begin{array}{ll}
&\dot{x}_{1}(t)=\sum_{j=1}^{n}a_{ij}x_{j}(t)-\epsilon [x_{1}(t)-s],\\
&\dot{x}_{i}(t)=\sum_{j=1}^{n}a_{ij}x_{j}(t),~~i=2,\cdots,n\end{array}\right.
\end{align}
can make every state $x_{i}(t)$ converge to $s$.

It is worth noticing that the above mentioned works all consider
continuous time feedback controllers and the disadvantage of such
controllers lies in that the controller must be imposed at every
time $t$. So it is not applicable to systems which can not endure
continuous disturbances. One can ask if we can pin the network only
at a very sparse time sequence to make every state $x_{i}(t)$
converge to $s$ for the consensus algorithm (\ref{eqnPulsea}).

Actually, to avoid such disadvantages, some discontinuous control
schemes, such as act-and-wait concept
control\cite{Insperger_2006,Insperger_2007}, intermittent
control\cite{CaoJD_2009,LiuXW_TNN_2011} and impulsive
technique\cite{YangXu_2005,LiuXZ_1994,LiuXZ_1996,Chua_TCAS_1997,YangT_2001}
have already been developed and used in the control of dynamical
systems. Particularly, in recent years, impulsive technique has been
successfully used in many areas such as neural
networks\cite{YangXu_2005}, control of spacecraft\cite{LiuXZ_1996},
secure communications\cite{Chua_TCAS_1997} and so on.

Compared to continuous-time controllers, impulsive controllers have
some obvious advantages. First, we only need to impose controllers
at a very sparse sequence of time points. Besides, it is typically
simpler and easier to implement. Recently, impulsive control
techniques have been used in the controlled synchronization and
consensus of complex networks. For example, in
\cite{GuanZH_TCAS_2010}, an impulsive distributed control scheme was
proposed to synchronize dynamical complex networks with both system
delay and multiple coupling delays. In \cite{XuZY_2008}, impulsive
control technique has been used in projective synchronization of
drive-response networks of coupled chaotic systems. In
\cite{TangY_Neurocomputing_2010}, the authors used impulsive control
technique to synchronize stochastic discrete-time networks. In
\cite{JHill_SICON_2011}, the authors proposed an impulsive hybrid
control scheme for the consensus of a network with nonidentical
nodes. Yet in these works, the controllers are imposed on all the
nodes of the networks. To take advantage of both the impulsive and
pinning control techniques, impulsive pinning technique has been
proposed which combines these two control techniques as a whole.
That is, the impulsive controllers are imposed only on a small
fraction of the nodes. For example, in
\cite{Zhou_TCAS_2011,CaoJD_TNN_2012}, impulsive pinning control
technique is used to stabilize and synchronize complex networks of
dynamical systems. In this paper, we will introduce this technique
into the pinning consensus algorithm. We show if the underlying
graph has spanning trees, then a single impulsive controller imposed
on one root is able to drive the network to reach consensus on a
given value when the impulsive strength is in a permissible range
and the impulse is sparse enough.

The rest of the paper is organized as follows. In Section
\ref{secPreliminaries}, some mathematical preliminaries are
presented; In Section \ref{secMainResults}, the sufficient
conditions for pinning consensus via one impulsive controller on
strongly connected graphs are proposed and proved; The results are
extended to graphs with spanning trees in Section
\ref{secGeneralResults}; Examples with numerical simulations are
provided in Section \ref{secNumericalSimulation} to illustrate the
theoretical results; And the paper is concluded in Section
\ref{secConclusion}.

\section{Mathematical preliminaries}\label{secPreliminaries}
In this section, we present some notations, definitions and lemmas
concerning matrix and graph theory that will be used later.

First, we introduce following definitions and notations from
\cite{Lu2006}.


\begin{definition} 
Suppose $A=[a_{ij}]_{i,j=1}^{n}\in \Rnn$. If
\begin{enumerate}
\item $a_{ij}\geq 0, \quad i\ne j,\qquad
a_{ii}=-\sum\limits_{j=1,j\neq i}^{n}a_{ij}, \quad i=\oneton$;

\item real parts of eigenvalues of $A$ are all negative except an
eigenvalue $0$ with multiplicity $1$,
\end{enumerate}
then we say $A\in{\bf A}_{1}$.
\end{definition}

\begin{definition}
Suppose $A=[a_{ij}]_{i,j=1}^{n}\in \Rnn$. If
\begin{enumerate}
\item $a_{ij}\geq 0, \quad i\ne j,\qquad
a_{ii}=-\sum\limits_{j=1,j\neq i}^{n}a_{ij}, \quad i=\oneton$;

\item $A$ is irreducible.
\end{enumerate}
Then we say $A\in{\mathbf A}_{2}$.
\end{definition}

It is clear that $\mathbf {A}_{2}\subseteq \mathbf{A}_{1}$.

By Gersgorin theorem  and Perron Frobenius theorem, we have the
following result.

\begin{lemma}\cite{Lu2006}.\label{lemMatrix}
If $A\in \mathbf {A}_{1}$, then the following items are valid:
\begin{enumerate}
\item If $\lambda$ is an eigenvalue of $A$ and $\lambda\neq 0$,
then  $Re(\lambda)<0$;

 \item $A$ has an eigenvalue $0$ with
multiplicity 1 and the right eigenvector $[1,1,\dots,1]^{\top}$;

\item  Suppose $\xi=[\xi_{1},\xi_{2},\cdots,\xi_{n}]^{\top}\in
\Rn$ (without loss of generality, assume
$\sum\limits_{i=1}^{n}\xi_{i}=1$) is the left eigenvector of $A$
corresponding to eigenvalue $0$. Then, $\xi_{i}\ge 0$ holds for all
$i=\oneton$; more precisely,

\item $A\in\mathbf{A}_{2}$ if and only if  $\xi_{i}>0$ holds for all
$i=\oneton$;

\item $A$ is reducible if and only if for some $i$, $\xi_{i}=0$.
In such case, by suitable rearrangement, assume that
$\xi^{\top}=[\xi_{+}^{\top},\xi_{0}^{\top}]$, where
$\xi_{+}=[\xi_{1},\xi_{2},\cdots,\xi_{p}]^{\top}\in \mathbb{R}^{p}$,
with all $\xi_{i}>0$, $i=1,\cdots,p$, and
$\xi_{0}=[\xi_{p+1},\xi_{p+2},\cdots,\xi_{n}]^{\top}\in
\mathbb{R}^{n-p}$ with all $\xi_{j}=0$, $p+1\le j\le n$. Then, $A$
can be rewritten as $\left[\begin{array}{cc}A_{11}& A_{12}\\A_{21}&
A_{22}\end{array}\right]$  where $A_{11}\in \mathbb{R}^{p\times p}$
is irreducible and $A_{12}=0$.
\end{enumerate}
\end{lemma}

\begin{remark}
By Lemma 1, for $A\in \mathbf{A}_{2}$, let
$\Xi=\diag[\xi_{1},\cdots,\xi_{n}]$ be the diagonal matrix generated
by the left eigenvector of $A$ corresponding to the eigenvalue $0$.
Then $\Xi A+A^{\top}\Xi\in \mathbf{A}_{2}$ is symmetric. Therefore,
its eigenvalues are real and satisfy $0=\lambda_{1}>\lambda_{2}\ge
\lambda_{3}\ge\cdots\ge \lambda_{n}$.
\end{remark}

A {\em weighted directed graph} of order $n$ is denoted by a triple
$\{\mathcal{V},\mathcal{E},A\}$, where
$\mathcal{V}=\{v_{1},\cdots,v_{n}\}$ is the vertex set,
$\mathcal{E}\subseteq \mathcal{V}\times \mathcal{V}$ is the edge
set, i.e., $e_{ij}=(v_{i},v_{j})\in \mathcal{E}$ if and only if
there is an edge from $v_{i}$ to $v_{j}$, and $A=[a_{ij}]$,
$i,j=1,\cdots, n$, is the weight matrix which is a nonnegative
matrix such that for $i,j\in\{1,\cdots,n\}$, $a_{ij}>0$ if and only
if $i\ne j$ and $e_{ji}\in\mathcal{E}$. For a weighted directed
graph $\mathcal{G}$ of order $n$, the graph Laplacian
$L(\mathcal{G})=[l_{ij}]_{i,j=1}^{n}$ can be defined from the weight
matrix $A$ in the following way:
\begin{eqnarray*}
l_{ij}=\left\{
\begin{array}{cc}
-a_{ij} & i\ne j\\
\sum\limits_{k=1,k\ne i}^{n}a_{ik}& j=i.
\end{array}
\right.
\end{eqnarray*}
A {\em(directed)  path} of length $l$ from vertex $v_{i}$ to $v_{j}$
is a sequence of $l+1$ distinct vertices $v_{r_{1}},\cdots,
v_{r_{l+1}}$ with $v_{r_{1}}=v_{i}$ and $v_{r_{l+1}}=v_{j}$ such
that $(v_{r_{k}},v_{r_{k+1}})\in \mathcal{E}(\mathcal{G})$ for
$k=1,\cdots,l$. A graph $\mathcal{G}$ is strongly connected if for
any two vertices $v$ and $w$ of $\mathcal{G}$, there is a directed
path from $v$ to $w$. A graph $\mathcal{G}$ contains a {\em spanning
(directed) tree} if there exists a vertex $v_{i}$ such that for all
other vertices $v_{j}$ there's a directed path from $v_{i}$ to
$v_{j}$, and $v_{i}$ is called the {\em root}.
\begin{remark}
From graph theory, a graph is strongly connected if and only if its
graph Laplacian $L$ satisfies $-L\in \mathbf{A}_{2}$.
\end{remark}

\section{Pinning consensus on strongly connected graphs}\label{secMainResults}

Consider the following consensus algorithm with a single impulsive
controller:
\begin{eqnarray}\label{eqnPulse}
\left\{
\begin{array}{ll}
\dot{x}(t)=-L x(t),&t\ne t_{k},\\
\Delta x_{r}(t_{k})=b_{k}[s-x_{r}(t_{k}^{-})],\\
\Delta x_{i}(t_{k})=0,& i\ne r. \end{array}\right.~~k=0,1,2,\cdots
\end{eqnarray}
where $L=[l_{ij}]$ is the graph Laplacian of the underlying graph,
$b_{k}$ is the strength of the impulse at time $t_{k}$, and
$0=t_{0}<t_{1}<t_{2}<\cdots$.

Without loss of generality, in the following, we always assume $s=0$ (by letting
$y_{i}(t)=x_{i}(t)-s$ and consider the new system of $y$) and $r=1$
(by suitable rearrangement when necessary). In this case, what we
need to do is to prove
\begin{eqnarray}
\lim_{t\rightarrow\infty}x_{i}(t)=0,~~ i=1,\cdots,n
\end{eqnarray}
for the following system
\begin{eqnarray}\label{eqnPulse1}
\left\{
\begin{array}{ll}
\dot{x}_{i}(t)=-\sum_{j=1}^{n}l_{ij}x_{j}(t),i=1,\cdots,n,~t\ne t_{k},\\
 x_{1}(t_{k}^{+})=(1-b_{k})x_{1}(t_{k}^{-}), \\
 x_{i}(t_{k}^{+})=x_{i}(t_{k}^{-}), i=2,3,\cdots,n. \end{array}\right.
\end{eqnarray}

Given $x(t)=[x_{1}(t),\cdots,x_{n}(t)]^{\top}$, denote
\begin{eqnarray}
\bar{x}(t)=\sum_{i=1}^{n}\xi_{i}x_{i}(t),
\end{eqnarray}
where $[\xi_{1},\cdots,\xi_{n}]^{\top}$ is the left eigenvector of
$L$ corresponding to the eigenvalue $0$ satisfying
$\sum_{i=1}^{n}\xi_{i}=1$, and
 $\Delta
t_{k}=t_{k+1}-t_{k}$, $k=0,1,2,3,\cdots$.

We also define the following Lyapunov function
\begin{eqnarray}
V(x(t))=\sum_{i=1}^{n}\xi_{i}[x_{i}(t)-\bar{x}(t)]^{2}.
\end{eqnarray}

\begin{remark}
Quantity $\bar{x}(t)$ and function $V(x(t))$ were introduced in
\cite{Lu2006} to discuss synchronization.
$\bar{X}(t)=[\bar{x}^{\top}(t),\cdots,\bar{x}^{\top}(t)]^{\top}$ is
the non-orthogonal projection of
$[x_{1}^{\top}(t),\cdots,x_{n}^{\top}(t)]^{\top}$ on the
synchronization manifold ${\mathcal S}=\{[x_{1}^{\top},\cdots,
x_{n}^{\top}]^{\top}\in \mathbb{R}^{nm}:~x_{i}=x_{j},~~
i,j=1,\cdots,n\}$, where
$x_{i}=[x_{i}^{1},\cdots,x_{i}^{m}]^{\top}\in \mathbb{R}^{m}$,
$i=1,\cdots,n$, and $x_{i}^{\top}$ represents the transpose of
$x_{i}$. $V(t)$ is some distance from
$[x_{1}^{\top}(t),\cdots,x_{n}^{\top}(t)]^{\top}$ to the
synchronization manifold ${\mathcal S}$. And synchronization is
equivalent to the distance goes to zero when time $t$ goes to
infinity, i.e.,
\begin{eqnarray}
\lim_{t\rightarrow\infty}V(t)=0.
\end{eqnarray}

\end{remark}

With the two functions $\bar{x}(t)$¡¡and $V(t)$, we will prove the
system with one impulsive controller (\ref{eqnPulse1}) can reach
consensus on $0$ by proving
\begin{eqnarray}
\lim_{t\rightarrow\infty}V(t)=0
\end{eqnarray}
and
\begin{eqnarray}
\lim_{t\rightarrow\infty}\bar{x}(t)=0
\end{eqnarray}
simultaneously.

The following theorem is the main result of this paper.

\begin{theorem}\label{thmMain}
Suppose $-L\in \mathbf{A}_{2}$, or equivalently, the underlying
graph is strongly connected, and there exist
$0<\eta_{1}\le\eta_{2}<1/\xi_{1}$ such that $b_{k}\in
[\eta_{1},\eta_{2}]$ for each $k$. If $\bar{x}(0)\ne 0$, then there
is a constant $T>0$ such that \eqref{eqnPulse1} will reach consensus
on $s$, when $\Delta t_{k}\ge T$ for each $k$.
\end{theorem}
\begin{remark}
It is interesting to note that the permissible range of the
impulsive strength is dependent on $\xi_{1}$ and decreasing with
$\xi_{1}$. Since in a strongly connected graph, $\xi_{1}<1$, we can
always choose $\eta_{2}>1$. Actually, in a network of $n$ nodes,
$\min_{i}\xi_{i}\le 1/n$. So, by properly choosing the pinning node,
we can always let $\eta_{2}>n$ except for the case $\xi_{i}=1/n$ for
each $i$, in which $\eta_{2}<n$ but can be arbitrarily close to $n$.
\end{remark}

The proof of Theorem \ref{thmMain} is divided into several steps.
First, we prove

\begin{lemma}\label{lemConvergenceV}
If $-L\in {\mathbf A}_{2}$, then
\begin{eqnarray}\label{Lemma2}
V(t_{k}^{-})\le
V(t_{k-1}^{+})e^{\frac{-\lambda_{2}}{\max_{i}\{\xi_{i}\}}\Delta
t_{k}},
\end{eqnarray}
where $\lambda_{2}>0$ is the smallest positive eigenvalue of the
symmetric matrix $\Xi L+L^{\top}\Xi$.
\end{lemma}
\begin{proof} Denote $\delta
x(t)=[x_{1}(t)-\bar{x}(t),\cdots,x_{n}(t)-\bar{x}(t)]^{\top}$. Then
\begin{align*}
\dot{V}(t)&=-2\sum_{i=1}^{n}\xi_{i}[x_{i}(t)-\bar{x}(t)]\big[\sum_{j=1}^{n}l_{ij}x_{j}(t)\big]\\
&=-2\sum_{i=1}^{n}\sum_{j=1}^{n}\xi_{i}l_{ij}[x_{i}(t)-\bar{x}(t)][x_{j}(t)-\bar{x}(t)]\\
&=-\delta x(t)^{\top}[\Xi L+L^{\top}\Xi]\delta x(t)\\
&\le-\lambda_{2}\|\delta x(t)\|^{2}\\
&\le\frac{-\lambda_{2}}{\max_{i}\{\xi_{i}\}}V(t).
\end{align*}
This implies \eqref{Lemma2}.
\end{proof}

\begin{remark}
By routine approach, it is desired to prove $V(t_{k}^{+})\le C
V(t_{k}^{-})$ for some constant $C$. Unfortunately, it is difficult
to prove it directly. Instead, we prove following Lemma.
\end{remark}

\begin{lemma}\label{lemMainEstimate}
Let $\epsilon,~\eta_{1},~\eta_{2}$ be constants satisfying
$0<\eta_{1}\le\eta_{2}<1/\xi_{1}$,
$0<\epsilon<\min\{\xi_{1},1/\eta_{2}-\xi_{1}\}$, the impulsive
strength $b_{k}\in[\eta_{1},\eta_{2}]$ for each $k$, $x(t)$ is a
solution of the system \eqref{eqnPulse1}. If
\begin{eqnarray}
\Delta t_{k}\ge\frac{\max_{i}\{\xi_{i}\}}{\lambda_{2}}\ln
\bigg(\frac{\xi_{1}}{\epsilon^{2}}
\frac{V(t_{k}^{+})}{\bar{x}^{2}(t_{k}^{+})}\bigg)\label{T_{k}},~
k=0,1,2,\cdots,
\end{eqnarray}
then, we have
\begin{eqnarray}\label{ineqJumpBar1}
|\bar{x}(t_{k+1}^{+})|\le
[1-\eta_{1}(\xi_{1}-\epsilon)]|\bar{x}(t_{k}^{+})|
\end{eqnarray}
and
\begin{eqnarray}\label{ineqCkEstimate}
\frac{V(t_{k+1}^{+})}{\bar{x}^{2}(t_{k+1}^{+})}\le
\frac{[2+4\eta_{2}^{2}(1-\xi_{1})]\epsilon^{2}/\xi_{1}+4\eta_{2}^{2}\xi_{1}(1-\xi_{1})}{[1-\eta_{2}(\xi_{1}+\epsilon)]^{2}}
\end{eqnarray}
for $k=0,1,2,\cdots$.
\end{lemma}
\begin{proof}
First, by (\ref{Lemma2}), we have
\begin{eqnarray}\label{21}
V(t_{k+1}^{-})\le
V(t_{k}^{+})e^{\frac{-\lambda_{2}}{\max_{i}\{\xi_{i}\}} \Delta
t_{k}},
\end{eqnarray}
which implies 
\begin{align}
|x_{1}(t_{k+1}^{-})-\bar{x}(t_{k+1}^{-})|\le
\sqrt{V(t_{k+1}^{-})/\xi_{1}}\le
\frac{\epsilon}{\xi_{1}}|\bar{x}(t_{k+1}^{-})|.
\end{align}
By \eqref{eqnPulse1}, we have
\begin{align}
\bar{x}(t_{k+1}^{+})
&=\bar{x}(t_{k+1}^{-})-b_{k+1}\xi_{1}x_{1}(t_{k+1}^{-})\nonumber\\
&=(1-b_{k+1}\xi_{1})\bar{x}(t_{k+1}^{-})\nonumber\\&
+b_{k+1}\xi_{1}[\bar{x}(t_{k+1}^{-})-x_{1}(t_{k+1}^{-})].
\end{align}
Thus, for $k=0,1,2,\cdots$,
\begin{eqnarray}\label{ineqJumpBar}
\left\{
\begin{array}{ll}
&|\bar{x}(t_{k+1}^{+})|
\ge [1-b_{k+1}(\xi_{1}+\epsilon)]|\bar{x}(t_{k+1}^{-})|\\
&|\bar{x}(t_{k+1}^{+})| \le
[1-b_{k+1}(\xi_{1}-\epsilon)]|\bar{x}(t_{k+1}^{-})|\end{array}\right.,
\end{eqnarray}
which implies
\begin{eqnarray}\label{ineqXiEsatimate1}
\left\{
\begin{array}{ll}
&|\bar{x}(t_{k+1}^{+})|
\ge [1-\eta_{2}(\xi_{1}+\epsilon)]|\bar{x}(t_{k+1}^{-})|\\
&|\bar{x}(t_{k+1}^{+})| \le
[1-\eta_{1}(\xi_{1}-\epsilon)]|\bar{x}(t_{k+1}^{-})|\end{array}\right..
\end{eqnarray}

Noting $\bar{x}(t_{k}^{+})=\bar{x}(t_{k+1}^{-})$, we have
\begin{align}\label{ineqXiEsatimate}
|\bar{x}(t_{k+1}^{+})| \le
[1-\eta_{1}(\xi_{1}-\epsilon)]|\bar{x}(t_{k}^{+})|,
\end{align}
which is just the inequality \eqref{ineqJumpBar1}.

On the other hand, noting the fact that
$\bar{x}^{2}(t_{k+1}^{-})=\bar{x}^{2}(t_{k}^{+})$ and (\ref{T_{k}}),
we have
\begin{align}
\frac{V(t_{k+1}^{-})}{\bar{x}^{2}(t_{k+1}^{-})}& \le
\frac{V(t_{k}^{+})e^{\frac{-\lambda_{2}}{\max_{i}\{\xi_{i}\}} \Delta
t_{k}}}{\bar{x}^{2}(t_{k}^{+})}=\frac{\epsilon^{2}}{\xi_{1}}\label{lemma3a}
\end{align}
Furthermore, by the assumption $\sum_{j=1}^{n}\xi_{j}=1$ and
inequality $(a+b)^{2}\le 2(a^{2}+b^{2})$, we have
\begin{align}\label{23}
&V(t_{k+1}^{+})=\xi_{1}[x_{1}(t_{k+1}^{+})-\bar{x}(t_{k+1}^{+})]^{2}\nonumber\\
&
+\sum_{i=2}^{n}\xi_{i}[x_{i}(t_{k+1}^{+})-\bar{x}(t_{k+1}^{+})]^{2}\nonumber\\
&=\xi_{1}[x_{1}(t_{k+1}^{-})-\bar{x}(t_{k+1}^{-})-b_{k+1}(1-\xi_{1})x_{1}(t_{k+1}^{-})]^{2}\nonumber\\
&
+\sum_{i=2}^{n}\xi_{i}[x_{i}(t_{k+1}^{-})-\bar{x}(t_{k+1}^{-})+b_{k+1}\xi_{1}x_{1}(t_{k+1}^{-})]^{2}\nonumber\\
&\le
2\big\{\sum_{i=1}^{n}\xi_{i}[x_{i}(t_{k+1}^{-})-\bar{x}(t_{k+1}^{-})]^{2}\nonumber\\
&
+b_{k+1}^{2}\xi_{1}(1-\xi_{1})^{2}x_{1}^{2}(t_{k+1}^{-})
+b_{k+1}^{2}\xi_{1}^{2}\sum_{i=2}^{n}\xi_{i}x_{1}^{2}(t_{k+1}^{-})\big\}\nonumber\\
&=2V(t_{k+1}^{-})+2b_{k+1}^{2}\xi_{1}(1-\xi_{1})x_{1}^{2}(t_{k+1}^{-})\nonumber\\
&\le
2V(t_{k+1}^{-})+2\eta_{2}^{2}\xi_{1}(1-\xi_{1})x_{1}^{2}(t_{k+1}^{-}).
\end{align}

By \eqref{ineqXiEsatimate1} and \eqref{23}, we have
\begin{align}
&\frac{V(t_{k+1}^{+})}{\bar{x}^{2}(t_{k+1}^{+})}
\le\frac{2V(t_{k+1}^{-})+2\eta_{2}^{2}\xi_{1}(1-\xi_{1})x_{1}^{2}(t_{k+1}^{-})}
{[1-\eta_{2}(\xi_{1}+\epsilon)]^{2}\bar{x}^{2}(t_{k+1}^{-})}\nonumber\\
&\le\frac{2V(t_{k+1}^{-})}{[1-\eta_{2}(\xi_{1}+\epsilon)]^{2}\bar{x}^{2}(t_{k+1}^{-})}\nonumber\\
&
+\frac{4\eta_{2}^{2}\xi_{1}(1-\xi_{1})\{[x_{1}(t_{k+1}^{-})-\bar{x}(t_{k+1}^{-})]^{2}
+\bar{x}^{2}(t_{k+1}^{-})\}}{[1-\eta_{2}(\xi_{1}+\epsilon)]^{2}\bar{x}^{2}(t_{k+1}^{-})}\nonumber\\
&\le\frac{[2+4\eta_{2}^{2}(1-\xi_{1})]V(t_{k+1}^{-})+4\eta_{2}^{2}\xi_{1}(1-\xi_{1})\bar{x}^{2}(t_{k+1}^{-})}
{[1-\eta_{2}(\xi_{1}+\epsilon)]^{2}\bar{x}^{2}(t_{k+1}^{-})}\nonumber\\
&\le\frac{[2+4\eta_{2}^{2}(1-\xi_{1})]\epsilon^{2}/\xi_{1}+4\eta_{2}^{2}
\xi_{1}(1-\xi_{1})}{[1-\eta_{2}(\xi_{1}+\epsilon)]^{2}}.
\end{align}
This proves \eqref{ineqCkEstimate}.
\end{proof}

From Lemma \ref{lemMainEstimate}, we can directly have the following
corollary.
\begin{corollary}\label{corMainEstimate}
Let $\epsilon,~\eta_{1},~\eta_{2}$ be constants satisfying
$0<\eta_{1}\le\eta_{2}<1/\xi_{1}$,
$0<\epsilon<\min\{\xi_{1},1/\eta_{2}-\xi_{1}\}$,
\begin{eqnarray}
C=\frac{[2+4\eta_{2}^{2}(1-\xi_{1})]\epsilon^{2}/\xi_{1}
+4\eta_{2}^{2}\xi_{1}(1-\xi_{1})}{[1-\eta_{2}(\xi_{1}+\epsilon)]^{2}},
\end{eqnarray}
For any given initial value $\bar{x}(0)\ne 0$, let
\begin{eqnarray*}
T=\frac{\max_{i}\{\xi_{i}\}}{\lambda_{2}}\bigg[\max\{\ln C,\ln
[V(0)/\bar{x}^{2}(0)]\}+\ln \frac{\xi_{1}}{\epsilon^{2}}\bigg],
\end{eqnarray*}
and $\Delta t_{k}\ge T$ for each $k$, then
$$|\bar{x}(t_{k}^{+})|\le
[1-\eta_{1}(\xi_{1}-\epsilon)]|\bar{x}(t_{k-1}^{+})|,~k=1,2,3,\cdots.$$
\end{corollary}

Now we can give the proof of Theorem \ref{thmMain}.

\begin{proof}First, since $\eta_{2}<1/\xi_{1}$, we can
choose $0<\epsilon<\min\{\xi_{1},1/\eta_{2}-\xi_{1}\}$. From
\eqref{ineqXiEsatimate}, we have
\begin{eqnarray}
|\bar{x}(t_{k}^{+})|\le
[1-\eta_{1}(\xi_{1}-\epsilon)]^{k}|\bar{x}(0)|
\end{eqnarray}
which implies
\begin{eqnarray*}
\lim_{k\to\infty}\bar{x}(t_{k}^{+})=0
\end{eqnarray*}
since $1-\eta_{1}(\xi_{1}-\epsilon)<1$. Combining the fact that
$\bar{x}(t)$ is constant on each $(t_{k},t_{k+1})$, we have
\begin{eqnarray*}
\lim_{t\to\infty}\bar{x}(t)=0.
\end{eqnarray*}

On the other hand, from Corollary \ref{corMainEstimate}, let $\Delta
t_{k}\ge T$, we have
\begin{eqnarray*}
V(t_{k}^{+})\le C\bar{x}^{2}(t_{k}^{+}),
\end{eqnarray*}
which leads to
\begin{eqnarray*}
\lim_{k\to\infty}V(t_{k}^{+})=0.
\end{eqnarray*}

Since on each $(t_{k},t_{k+1})$,
\begin{eqnarray*}
V(t)\le V(t_{k}^{+})e^{\frac{-\lambda_{2}}{\max_{i}\{\xi_{i}\}}
(t-t_{k})},
\end{eqnarray*}
this also implies
\begin{eqnarray*}
\lim_{t\to\infty}V(t)=0
\end{eqnarray*}
and
\begin{eqnarray*}
\lim_{t\to\infty}[x_{i}(t)-\bar{x}(t)]=0.
\end{eqnarray*}

Thus,
\begin{eqnarray*}
\lim_{t\to\infty}x_{i}(t)=\lim_{t\to\infty}[x_{i}(t)-\bar{x}(t)]+\lim_{t\to\infty}\bar{x}(t)
=0.
\end{eqnarray*}
The proof is completed.
\end{proof}
\begin{remark}
In \cite{Zhou_TCAS_2011}, Zhou et.al discussed pinning complex
delayed dynamical networks by a single impulsive controller. In that
paper, the authors proposed a novel model. However, the coupling
matrix $A$ is assumed to be a symmetric irreducible matrix and
orthogonal eigen-decomposition is used and plays a key role.
Therefore, the approach can not apply to our case.
\end{remark}

\begin{remark}
In \cite{CaoJD_TNN_2012}, Lu et.al, discussed synchronization
control for nonlinear stochastic dynamical networks by impulsive
pinning strategy. In that strategy, at each impulse time point
$t_{k}$, the authors select several nodes with largest errors, and
adding controllers to those nodes. Therefore, one needs to observe
all states $x_{i}(t_{k})$ at each $t_{k}$. In our strategy, we only
need to know the state $x_{1}(t_{k})$ and one controller is enough.
\end{remark}

From Theorem \ref{thmMain}, we can have the following corollary in
the case that the impulsive strength is a constant.
\begin{corollary}
Suppose $-L\in \mathbf{A}_{2}$, or equivalently, the underlying
graph is strongly connected, and $b_{k}=b\in (0,1/\xi_{1})$ for each
$k$. If $\bar{x}(0)\ne 0$, then there exists a constant $T>0$ such
that \eqref{eqnPulse} will reach consensus on $s$ when $\Delta
t_{k}\ge T$ for each $k$.
\end{corollary}

\section{Pinning consensus on graphs with spanning
trees}\label{secGeneralResults} In this section, we will generalize
the results obtained in previous section to graphs with spanning
trees. In such case, by suitable arrangement, we can assume that $L$
has the following $m\times m$ block form:
\begin{eqnarray}\label{eqnLaplaceReducible}
L=\left[
\begin{array}{ccccc}
L_{11} & 0&0&\cdots&0\\
L_{21} & L_{22}&0&\cdots&0\\
\vdots&\cdots&\ddots&\cdots&0\\
L_{m1}&\cdots&\cdots&\cdots&L_{mm}
\end{array}\right]
\end{eqnarray}
where $-L_{ii}\in \mathbb{R}^{p_{i}\times p_{i}}$ is irreducible,
and $[L_{i1},\cdots,L_{i(i-1)}]\ne 0$ for $i=2,\cdots,m$. Let
$[\xi_{1},\cdots,\xi_{n}]^{\top}$ be the normalized left eigenvector
of $L$ corresponding to the eigenvalue $0$. From Lemma
\ref{lemMatrix}, $\xi_{i}>0$ for $i=1,\cdots,p_{1}$, and $\xi_{i}=0$
for $i=p_{1}+1,\cdots,n$. Thus $\bar{x}(t)=\sum_{i=1}^{p_{1}}\xi_{i}
x_{i}(t)$.

We will prove
\begin{theorem}\label{thmGeneral}
Suppose the underlying graph is of the form
(\ref{eqnLaplaceReducible}), and there exist
$0<\eta_{1}\le\eta_{2}<1/\xi_{1}$ such that $\eta_{1}\le
b_{k}\le\eta_{2}$ for each $k$. If $\bar{x}(0)\ne 0$, then the
consensus algorithm \eqref{eqnPulse1} can reach consensus on a given
value $s$ when $\Delta t_{k}\ge T$ for a large enough $T$.
\end{theorem}
\begin{proof}
Let $x(t)=[X_{1}^{\top}(t),\cdots, X_{m}^{\top}(t)]^{\top}$ with
$X_{i}(t)=[x_{m_{i}+1}(t),\cdots,x_{m_{i+1}}(t)]^{\top}$, where
$m_{1}=0$ and $m_{i+1}=m_{i}+p_{i}$. Since
$\bar{x}(0)=\sum_{i=1}^{p_{1}}\xi_{i}x_{i}\ne 0$, by applying
Theorem \ref{thmMain} to the subsystem of $X_{1}(t)$, we can find
$T>0$ such that if $\Delta t_{k}\ge T$ for each $k$, then
\begin{eqnarray*}
\lim_{t\to\infty}x_{i}(t)=0,\quad i=1,\cdots,p_{1}.
\end{eqnarray*}

Consider the subsystem of $X_{2}(t)$, we have:
\begin{eqnarray}\label{eqnPart2}
\dot{X}_{2}(t)=-L_{22}X_{2}(t)-L_{21}X_{1}(t).
\end{eqnarray}

Denote $Y_{2}(t)=-L_{21}X_{1}(t)$. Then \eqref{eqnPart2} can be
rewritten as:
\begin{eqnarray}
\dot{X}_{2}(t)=-L_{22}X_{2}(t)+Y_{2}(t)
\end{eqnarray}
Thus,
\begin{eqnarray}\label{ineqX2Estimate}
X_{2}(t)=e^{-L_{22}t}X_{2}(0)+\int_{0}^{t}e^{-L_{22}(t-s)}Y_{2}(s)ds.
\end{eqnarray}
Since the $L_{21}\ne 0$, at least one row sum of $L_{22}$  is
negative, which implies that $L_{22}$ is a non-singular M-matrix and
its eigenvalues $\mu_{1}$, $\cdots$, $\mu_{p_{2}}$ can be arranged
as $0<Re(\mu_{1})\le\cdots\le Re(\mu_{p_{2}})$. Then,
\begin{eqnarray*}
\|e^{-L_{22}t}\|\le Ke^{-Re(\mu_{1})t}
\end{eqnarray*}
for some constant $K>0$. And
\begin{eqnarray*}
\|X_{2}(t)\|&\le&
K\|X_{2}(0)\|e^{-Re(\mu_{1})t}\\
&&+K\int_{0}^{t}e^{-Re(\mu_{1})(t-s)}\|Y_{2}(s)\|ds
\end{eqnarray*}
It is obvious that
\begin{eqnarray*}
\lim_{t\rightarrow\infty}K\|X_{2}(0)\|e^{-Re(\mu_{1})t}=0
\end{eqnarray*}

To show
$$\lim_{t\to\infty}\|X_{2}(t)\|=0,$$ we only need to estimate the
second term on the righthand side of \eqref{ineqX2Estimate}.

Since $\lim_{t\to\infty}\|Y_{2}(t)\|=0$, for any $\epsilon>0$, there
exists $t_{\epsilon}>0$ such that $\|Y_{2}(t)\|\le \epsilon$ for
each $t\ge t_{\epsilon}$. Furthermore, $Y_{2}(t)$ is uniformly
bounded. Let $\overline{Y}_{2}>0$ be an upper bound of $Y_{2}(t)$.
Then for
$t>t_{\epsilon}+\displaystyle\frac{1}{Re(\mu_{1})}\ln\frac{\overline{Y}_{2}}{\epsilon}$,
\begin{align*}
&\int_{0}^{t}e^{-Re(\mu_{1})(t-s)}\|Y_{2}(s)\|ds
=\int_{0}^{t_{\epsilon}}e^{-Re(\mu_{1})(t-s)}\|Y_{2}(s)\|ds\\
&+\int_{t_{\epsilon}}^{t}e^{-Re(\mu_{1})(t-s)}\|Y_{2}(s)\|ds\\
&\le\overline{Y}_{2}\int_{0}^{t_{\epsilon}}e^{-Re(\mu_{1})(t-s)}ds
+\epsilon\int_{t_{\epsilon}}^{t}e^{-Re(\mu_{1})(t-s)}ds\\
&=\frac{\overline{Y}_{2}}{Re(\mu_{1})}e^{-Re(\mu_{1})t}[e^{Re(\mu_{1})t_{\epsilon}}-1]\\&
+\frac{\epsilon}{Re(\mu_{1})}[1-e^{-Re(\mu_{1})(t-t_{\epsilon})}]\\
&=\frac{\overline{Y}_{2}}{Re(\mu_{1})}e^{-Re(\mu_{1})(t-t_{\epsilon})}[1-e^{-Re(\mu_{1})t_{\epsilon}}]\\&
+\frac{\epsilon}{Re(\mu_{1})}[1-e^{-Re(\mu_{1})(t-t_{\epsilon})}]\\
&\le\frac{\epsilon}{Re(\mu_{1})}[1-e^{-Re(\mu_{1})t_{\epsilon}}]
+\frac{\epsilon}{Re(\mu_{1})}\\
&\le\frac{2\epsilon}{Re(\mu_{1})}
\end{align*}
Because $\epsilon$ is arbitrary, we have
$$
\lim_{t\to\infty}\int_{0}^{t}e^{-Re(\mu_{1})(t-s)}\|Y_{2}(s)\|ds=0
$$
Thus,
$$
\lim_{t\to\infty}\|X_{2}(t)\|=0.
$$

For $i=3,\cdots,n$, we have
\begin{eqnarray*}
\dot{X}_{i}(t)=-L_{ii}X_{i}(t)-Y_{i}(t),
\end{eqnarray*}
where $Y_{i}(t)=\sum_{j=1}^{i-1}L_{ij}X_{j}(t)$.

By induction, if we already have
$$
\lim_{t\to\infty}\|X_{j}(t)\|=0
$$
for $j=1,\cdots,i-1$, then we have
\begin{eqnarray}
\lim_{t\to\infty}Y_{i}(t)=0.
\end{eqnarray}

By a similar analysis as above, we can show that
$$
\lim_{t\to\infty}\|X_{i}(t)\|=0.
$$
\end{proof}

Similarly, we can have a corollary from Theorem \ref{thmGeneral}
when the impulsive strength is constant.
\begin{corollary}
Suppose the underlying graph is of the form
(\ref{eqnLaplaceReducible}), and $b_{k}=b\in (0,1/\xi_{1})$ for each
$k$. If $\bar{x}(0)\ne 0$, then the consensus algorithm
\eqref{eqnPulse1} can reach consensus on a given value $s$ when
$\Delta t_{k}\ge T$ for a large enough $T$.
\end{corollary}

\section{Numerical Simulations}\label{secNumericalSimulation}
In this section we will provide two simple examples to illustrate
the theoretical results. The first example considers a strongly
connected graphs. And the second one concerns a graph that is not
strongly connected but has a spanning tree.

\subsection{Example 1}
In the first example, we consider a directed circular network. (Fig.
\ref{figNet1} shows an example of a circular network with $10$
nodes.) It is obvious that this network is strongly connected. If we
assign each edge with weight $1$, then the graph Laplacian is
\begin{eqnarray*}
L=\left[
\begin{array}{cccccccccc}
1&0&0&\cdots&0&0&0&-1\\
-1&1&0&\cdots&0&0&0&0\\
0&-1&1&\cdots&0&0&0&0\\
0&0&-1&\ddots&\vdots&\vdots&\vdots&\vdots\\
\vdots&\vdots&\vdots&\ddots&1&0&0&0\\
0&0&0&\cdots&-1&1&0&0\\
0&0&0&\cdots&0&-1&1&0\\
0&0&0&\cdots&0&0&-1&1\\
\end{array} \right]
\end{eqnarray*}

Then we have $\xi_{i}=0.01$ for each $i$, and
$\lambda_{2}=3.9465\times 10^{-5}$.
\begin{figure}
\centering
\includegraphics[width=0.5\textwidth]{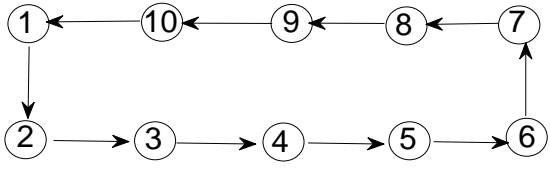}
\caption{A circular network consisting of $10$ nodes.}
\label{figNet1}
\end{figure}

Randomly choose an initial value $x(0)$ whose $\bar{x}(0)=0.4886$.
The objective is to drive the network to reach a consensus on value
$0$. After calculation, we have
$$V(0)=0.01\sum_{i=1}^{100}[x_{i}(0)-\bar{x}(0)]^{2}=0.5935,$$
$$V(0)/\bar{x}^{2}(0)=2.4856.$$
Let $b_{k}=11$ for each $k$, then we can set $\eta_{1}=\eta_{2}=11$.
Choose $\epsilon=0.00999$. Then,
$$C=\frac{[2+4\eta_{2}^{2}(1-\xi_{1})]\epsilon^{2}/\xi_{1}
+4\eta_{2}^{2}\xi_{1}(1-\xi_{1})}{[1-\eta_{2}(\xi_{1}+\epsilon)]^{2}}
=15.7641.$$

Then we get the lower bound for the duration between each successive
impulse is
$$T=\frac{\max_{i}\{\xi_{i}\}}{\lambda_{2}}\ln \frac{C\xi_{1}}{\epsilon^{2}}
=1.8662\times 10^{3}.$$

In the simulation, we set $\Delta t_{k}=1867$ for each $k$. The
simulation result is presented in Figs.\ref{figSimu1},\ref{figVar1}.
Fig.\ref{figSimu1} shows the trajectories of the network, and
Fig.\ref{figVar1} shows the variations of the trajectories with
respect to time $t$ which is defined as
$$\var(t)=\sum_{i=1}^{n}|x_{i}(t)|.$$

It can be seen that the network will asymptotically reach a
consensus on value $0$.
\begin{figure}
\centering
\includegraphics[width=0.5\textwidth,height=0.41\textwidth]{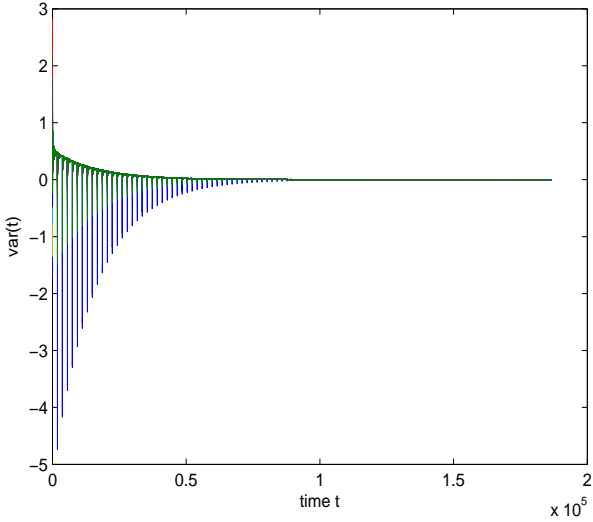}
\caption{Pinning consensus to $0$ on a circular network with $100$
nodes. \label{figSimu1}}
\end{figure}
\begin{figure}
\centering
\includegraphics[width=0.5\textwidth]{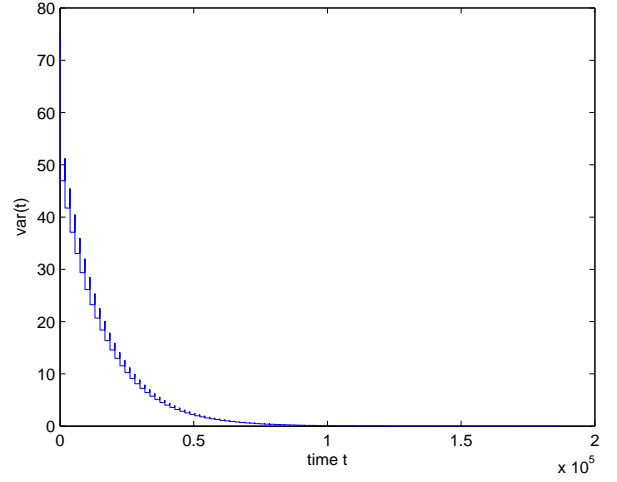}
\caption{The variation of the trajectories of the circular network
with $100$ nodes.} \label{figVar1}
\end{figure}

\subsection{Example 2}
In this example, we consider a network that is not strongly
connected but has a spanning tree. We start from a circular network
with $10$ nodes (shown in Fig.\ref{figNet1}) and construct a larger
network by randomly adding new nodes to the network. At each step,
randomly choose a node $i$ from the existing network, then a new
node $j$ is added to the network such that there is a directed edge
from $i$ to $j$. Continuing this procedure until the network has
$100$ nodes, we obtain a graph that has spanning trees but is not
strongly connected. If we assign each edge with weight $1$, then in
the graph Laplacian \eqref{eqnLaplaceReducible},
\begin{eqnarray*}
L_{11}=\left[
\begin{array}{cccccccccc}
1&0&0&0&0&0&0&0&0&-1\\
-1&1&0&0&0&0&0&0&0&0\\
0&-1&1&0&0&0&0&0&0&0\\
0&0&-1&1&0&0&0&0&0&0\\
0&0&0&-1&1&0&0&0&0&0\\
0&0&0&0&-1&1&0&0&0&0\\
0&0&0&0&0&-1&1&0&0&0\\
0&0&0&0&0&0&-1&1&0&0\\
0&0&0&0&0&0&0&-1&1&0\\
0&0&0&0&0&0&0&0&-1&1
\end{array} \right]
\end{eqnarray*}
Thus $\xi_{i}=0.1$ for $1\le i\le 10$, $\xi_{i}=0$ for $11\le i\le
100$, and $\lambda_{2}=0.3820$. Randomly choose the initial value
$x(0)$ where $\bar{x}(0)=0.3909$. The objective is to drive the
network to reach consensus on the value $0$. After calculation, we
have:
$$V(0)=\sum_{i=1}^{10}\xi_{i}[x_{i}(0)-\bar{x}(0)]^{2}=0.6369.$$
$$V(0)/\bar{x}^{2}(0)=4.1677.$$
Let $b_{k}=5$ for each $k$, then we can set $\eta_{1}=\eta_{2}=5$.
Choose $\epsilon=0.09$. Then we have
$$C=\frac{[2+4\eta_{2}^{2}(1-\xi_{1})]\epsilon^{2}/\xi_{1}
+4\eta_{2}^{2}\xi_{1}(1-\xi_{1})}{[1-\eta_{2}(\xi_{1}+\epsilon)]^{2}}
=7.3280.$$

Thus the lower bound for the intervals between each successive
impulse is
$$T=\frac{\max_{i}\{\xi_{i}\}}{\lambda_{2}}\ln \frac{C\xi_{1}}{\epsilon^{2}}
=14.8720.$$


In the simulation, we choose $\Delta t_{k}=15$. The simulation
result is presented in Figs.\ref{figSimu2}, \ref{figVar2}. It can be
seen that the network will asymptotically reach a consensus on $0$.
\begin{figure}
\centering
\includegraphics[width=0.5\textwidth]{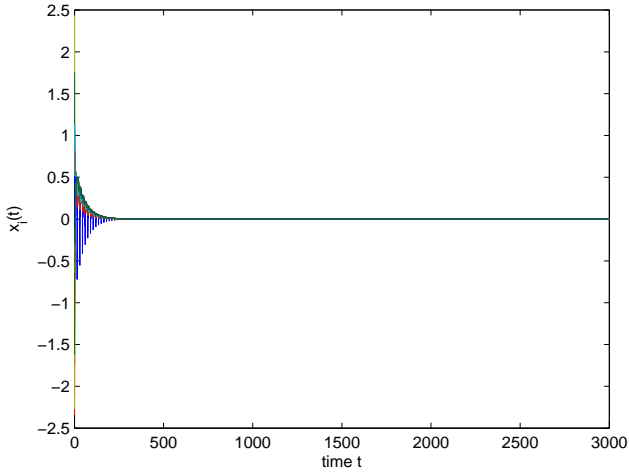}
\caption{Pinning consensus to $0$ on the graph that has spanning
trees. \label{figSimu2}}
\end{figure}
\begin{figure}
\centering
\includegraphics[width=0.5\textwidth]{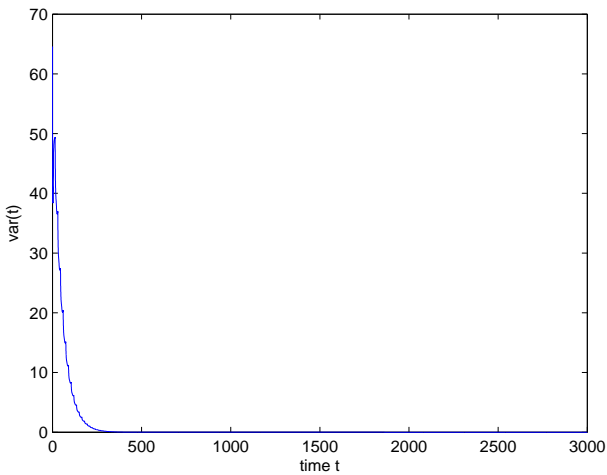}
\caption{The variation of the trajectories on a graph that has
spanning trees.} \label{figVar2}
\end{figure}
\section{Conclusions}\label{secConclusion}
In this paper, we investigate pinning consensus in networks of
multiagents via a single impulsive controller. First, we prove a
sufficient condition for a network with a strongly connected
underlying graph to reach consensus on a given value. Then we extend
the result to networks with a spanning tree. Interestingly, we find
the permissible range of the impulsive strength is determined by the
left eigenvector of the graph Laplacian corresponding to the zero
eigenvalue and the pinning node we choose. Besides, a sparse enough
impulsive pinning on one node can always drive the network to reach
consensus on a prescribed value. Examples with numerical simulations
are also provided to illustrate the theoretical results. The pinning
synchronization in complex networks via a single impulsive
controller is an interesting issue, which will be worked out soon.


\begin{thebibliography}{99}

\bibitem{Ren_2005}
W. Ren, R.W. Beard, and E.M. Atkins, "A survey of consensus problems
in multi-agent coordination", in {\it Proceedings of 2005 American
Control Conference}, pp. 1859-1864, 2005.

\bibitem{Olfati_2007}
R. Olfati-Saber, "Consensus and cooperation in networked multi-agent
systems", {\it Proceedings of the IEEE}, vol. 95, no. 1, pp.
215-233, 2007.

\bibitem{Lu2004}
W. Lu, T. Chen, ``Synchronization  of coupled connected neural
networks with delays", {\it IEEE Transactions on Circuits and
Systems-I, Regular Papers}, vol. 51, no. 12, pp. 2491-2503, 2004.

\bibitem{Lu2006}
W. Lu, T. Chen, "New approach to synchronization analysis of
linearly coupled ordinary differential systems", {\it Physica D},
vol. 213, pp. 214-230, 2006.

\bibitem{Pca}
T. Chen, S. Amari, and Q. Lin, "A unified algorithm for principal
and minor component extraction", {\it Neural Networks}, vol. 11, no.
3, pp.385-389, 1998.

\bibitem{Psa}
T. Chen, Y. Hua, and W. Yan, "Global convergence of oja's subspace
algorithm for principal component extraction", {\it IEEE
Transactions on Neural Networks}, vol.8, (1998) pp.57-68

\bibitem{Chen2007}
T. Chen, X. Liu, and W. Lu, "Pinning complex networks by a single
controller", {\it IEEE Transactions on Circuits and Systems-I:
Regular Papers}, vol. 54, no. 6, pp. 1317-1326, 2007.

\bibitem{Moore_2005}
K. Moore, and D. Lucarelli, "Forced and constrained consensus among
cooperating agents", in {\it IEEE international conference on
networking, sensing and control}, pp. 449-454, 2005.

\bibitem{YangXu_2005}
Z. Yang, and D. Xu, "Stability analysis of delay neural networks
with impulsive effects", {\it IEEE Transactions on Circuits and
Systems II-Express Briefs}, vol. 52, no. 8, pp. 517-521, 2005.

\bibitem{ChenFei_Automatica_2009}
F. Chen, Z. Chen, L. Xiang, Z. Liu, and Z. Yuan, ``Reaching a
consensus via pinning control", {\it Automatica}, vol. 45, pp.
1215-1220, 2009.

\bibitem{Insperger_2006}
T. Insperger, "Act-and-wait concept for continuous-time control
systems with feedback delay", {\it IEEE Transactions on Control
Systems Technology}, vol. 14, no. 5, pp. 974-977, 2006.

\bibitem{Insperger_2007}
T. Insperger and G. Stepan, "Act-and-wait control concept for
discrete-time systems with feedback delay", {\it IET Control Theory
and Applications}, vol. 1, no. 3, pp. 553-557, 2007.

\bibitem{CaoJD_2009}
W. Xia and J. Cao, "Pinning synchronization of delayed dynamical
networks via periodically intermittent control", {\it Chaos}, vol.
19, no. 1, 013120, 2009.

\bibitem{LiuXW_TNN_2011}
X. Liu and T. Chen, ``Cluster synchronization in directed networks
via intermittent pinning control", {\it IEEE Transactions on Neural
Networks}, vol. 22, no. 7, pp. 1009-1020, 2011.

\bibitem{LiuXZ_1994}
X. Liu, "Stability results for impulsive differential systems with
applications to population growth models", Dynamics and Stability of
Systems, vol. 9, no. 2, pp. 163-174, 1994.

\bibitem{LiuXZ_1996}
X. Liu, and A. Willms, "Impulsive controllability of linearly
dynamical systems with applications to maneuvers of spacecraft",
{\it Mathematical Problems in Engineering}, vol. 2, no. 4, pp.
277-299, 1996.

\bibitem{Chua_TCAS_1997}
T. Yang, and L. Chua, "Impulsive stabilization for control and
synchronization of chaotic systems: Theory and application to secure
communication", {\it IEEE Transactions on Circuits and Systems
I-Foundamental Theory and Applications}, vol. 44, no. 10, pp.
976-988, 1997.

\bibitem{YangT_2001}
T. Yang, {\it Impulsive Systems and Control: Theory and
Applications}. Huntington, NY: Nova Science, 2001.


\bibitem{JHill_SICON_2011}
B. Liu, and D. Hill, "Impulsive consensus for complex dynamical
networks with nonidentical nodes and coupling time-delays", {\it
SIAM Journal on Control and Optimization}, vol. 49, no. 2, pp.
315-338, 2011.

\bibitem{WangZD_TNN_2011}
W. Xiong, D. Ho, and Z. Wang, "Consensus analysis of multiagent
networks via aggregated and pinning approaches", {\it IEEE
Transactions on Neural Networks}, vol. 22, no. 8, pp. 1231-1240,
2011.

\bibitem{Zhou_TCAS_2011}
J. Zhou, Q. Wu, and L. Xiang, "Pinning complex delayed dynamical
networks by a single impulsive controller", {\it IEEE Transactions
on Circuits and Systems-I: Regular Papers}, vol. 58, no. 12, pp.
2882-2893, 2011.

\bibitem{CaoJD_TNN_2012}
J. Lu, J. Kurths, J. Cao, N. Mahdavi, and C. Huang, "Synchronization
control for nonlinear stochastic dynamical networks: pinning
impulsive strategy", {\it IEEE Transactions on Neural Networks and
Learning Systems}, vol. 23, no. 2, pp. 285-292, 2012.

\bibitem{XuZY_2008}
L. Guo, Z. Xu, and M. Hu, "Projective synchronization in
drive-response networks via impulsive control", {\it Chinese Physics
Letters}, vol. 25, no. 8, pp. 2816-2819, 2008.


\bibitem{GuanZH_TCAS_2010}
Z. Guan, Z. Liu, G. Feng, and Y. Wang, "Synchronization of complex
dynamical networks with time-varying delays via impulsive
distributed control", {\it IEEE Transactions on Circuits and
Systems-I: Regular Papers}, vol. 57, no. 8, pp. 2182-2195, 2010.

\bibitem{TangY_Neurocomputing_2010}
Y. Tang, S. Leung, W. Wong, and J. Fang, "Impulsive pinning
synchronization of stochastic discrete-time networks", {\it
Neurocomputing}, vol. 73, pp. 2132-2139, 2010.

\bibitem{Porfiri_Chaos_2011}
P. DeLellis, M. di Bernardo, and M. Porfiri, "Pinning control of
complex networks via edge snapping", {\it Chaos}, vol. 21, 033119,
2011.

\bibitem{SuH_TSMC_2012}
H. Su, Z. Rong, M. Chen, X. Wang, G. Chen, and H. Wang,
"Decentralized adaptive pinning control for cluster synchronization
of complex dynamical networks", {\it IEEE Transactions on Systems,
Man, and Cybernetics-Part B: Cybernetics}, 2012.
\end{thebibliography}
\end{document}